\documentclass[11pt,reqno]{amsart}
\usepackage{amscd,amssymb,amsmath,amsthm}
\usepackage{cite}
\newtheorem{thm}[subsection]{Theorem}

\newtheorem{pro}[subsection]{Proposition}
\newtheorem{cor}[subsection]{Corollary}
\newtheorem{problem}[subsection]{Problem}
\newtheorem{rk}[subsection]{Remark}

\newtheorem{ex}{Example}

\numberwithin{equation}{section} 

\newcommand{\bea}{\begin{eqnarray}}
\newcommand{\eea}{\end{eqnarray}}

\DeclareMathOperator{\Der}{\rm Der}

\def\g{\gamma}

\def\xb{{\mathbf{x}}}

\def\yb{{\mathbf{y}}}

\begin{document}

\begin{center}
{\Large {\bf On $\exp(Der(E))$ of nilpotent evolution
algebras}}\\[1cm]

{\sc Farrukh Mukhamedov}$^{1,2,*}$ \\[2mm]

$^1$ Department of Mathematical Sciences, College of Science, \\
 United Arab Emirates University 15551, Al-Ain,\\
 Abu Dhabi, United
Arab Emirates\\

e-mail: {\tt far75m@gmail.com; farrukh.m@uaeu.ac.ae}

$^*$ Corresponding author.\\[3mm]

{\sc Otabek Khakimov}$^{2,3}$ \\[2mm]

$^2$ Institute of Mathematics named after V.I.Romanovski, 4,\\
University str., 100125, Tashkent\\
$^3$ AKFA University, 1st Deadlock 10,\\
 Kukcha Darvoza, 100095 Tashkent, Uzbekistan \\

e-mail: {{\tt hakimovo@mail.ru};  {\tt o.khakimov@mathinst.uz}}\\[2mm]

{\sc  Izzat Qaralleh }$^{4}$ \\[2mm]

$^4$Department of Mathematics\\
Faculty of Science, Tafila Technical University\\
P.O. Box, 179, Tafila, Jordan\\

e-mail: {\tt izzat\_math@yahoo.com}\\[2mm]
\end{center}

\small
\begin{center}
{\bf Abstract}\\
\end{center}
In the present paper, every evolution algebra is endowed with Banach algebra norm. This together with the description of derivations and automorphisms of nilpotent evolution algebras, allows to investigated the set $\exp(Der(E))$. Moreover, it is proved that $\exp(Der(E))$ is  a normal subgroup of $Aut(E)$, and its corresponding index is calculated.

\normalsize

\section{Introduction}

Many papers were devoted in the algebraic formulation of Mendel's
laws in terms of non-associative algebras (see, for example
\cite{E,lu,Reed,2}). On the other hand, certain genetic phenomena
such as, for example, the case of incomplete dominance, systems of
multiple alleles, and asexual inheritance, do not follow
 Mendel's laws, therefore, in \cite{tv} it has introduced a new type of evolution algebra which was partly an attempt to
 study such non-Mendelian behavior.
The study of evolution algebras constitutes a new subject both in
algebra and the theory of dynamical systems. There are many
related open problems to promote further research in this subject
(for more details we refer to \cite{t}).

We notice that evolution algebras are not defined by identities, and therefore they do not form a variety of non-associative
algebras, like Lie, Jordan or alternative algebras. Hence, the investigation of such kind of algebras needs
a different approach (see \cite{Some_properties,derevol,rozomir}).
 A classification of low dimensional evolution
algebras have been carried out in\cite{3dim,heg1,heg2,Elduque,MKQ20,MQ21,ORV,QM21}.
Evolution algebras found their applications in models of non-Mendelian genetics
laws \cite{BBV,10,FFN17,RV19}. Moreover, these algebras are tightly connected with group theory, the theory of knots, dynamic systems, Markov processes and graph theory \cite{BMV,CRR19, CRR20,CNT}.
Evolution algebras allowed introduce useful algebraic techniques and methods into the investigation of some digraphs because such kind of algebras and weighted digraphs can be canonically identified \cite{Elduque,tv}
However, a full classification of nilpotent evolution algebras is
far from its solution. For review on recent development on evolution algebras, we refer the reader to \cite{CFNT22}.

Recently, in \cite{MV19,V19} it has been studied
algebra norms on evolution algebras. Here, by \textit{norm
algebra} we mean an evolution algebra $E$ endowed with a norm
$\|\cdot\|$ such that $\|ab\|\leq\|a\|\|b\|$, for every $a,b\in
E$, and $E$ is a \textit{Banach algebra} if it has a complete
algebra norm. Basically, in the mentioned papers, the norm was
taken as $\ell_1$-norm defined in terms of a fixed natural basis.
Furthermore, certain convergence of the trajectories of the
multiplication operator was investigated. However, it is natural to consider other types of Banach norm on
infinite dimensional evolution algebras. As an example, recently, in \cite{VCR20} Hilbert evolution algebras have been introduced and studied.

In the present paper, we endow every evolution algebra (finite-dimensional) with Banach algebra norm. We stress that the defined norm provides further directions in the infinite dimensional evolution algebras which involves $\ell_\infty$-norms (it will be a topic of another paper). On the other hand, recently, derivations and
automorphisms of nilpotent evolution algebras have been studied and described \cite{MKOQ19,MKQ20}. This together with Banach algebra structure allows us to investigated the set
$\exp(Der(E))$. Moreover, we prove that $\exp(Der(E))$ is a normal subgroup of $Aut(E)$, and its corresponding index is calculated.
The obtained results will shed some light into inner automorphism problem on evolution algebras.

\section{Evolution algebras}

Recall the definition of evolution algebras. Let $\bf{E}$ be a
vector space over a field $\mathbb K$. In what follows, we always
assume that $\mathbb K$ has characteristic zero. The vector space
$\bf{E}$ is called {\it evolution algebra} w.r.t. {\it natural
basis} $\{{\bf e}_1, {\bf e}_2, . . . \}$ if a multiplication rule
$\cdot$ on $\bf{E}$ satisfies
$$
{\bf e}_i\cdot {\bf e}_j={\bf 0},\ i\neq j,
$$
$$
{\bf e}_i\cdot{\bf e}_i=\sum_{k}a_{ik}{\bf e}_k,\ i\geq1.
$$

From the above definition it follows that evolution algebras are commutative (therefore, flexible).

We denote by $A=(a_{ij})^n_{i,j=1}$ the matrix of the structural constants
 of the finite-dimensional evolution algebra $\bf{E}$.
Obviously, $rank A =\dim(\bf{E}\cdot\bf{E})$. Hence, for finite-dimensional evolution algebra the
rank of the matrix does not depend on choice of natural basis.

In what follows for convenience, we write ${\bf u}{\bf v}$ instead ${\bf u}\cdot{\bf v}$
for any ${\bf u},{\bf v}\in\bf{E}$ and we shall write $\bf{E}^2$ instead $\bf{E}\cdot\bf{E}$.


For an evolution algebra $\bf E$ we introduce the following
sequence, $k\geq1$
\begin{equation}\label{E^k}
{\bf E}^{k}=\sum_{i=1}^{\lfloor k/2\rfloor}{\bf E}^{i}{\bf
E}^{k-i},
\end{equation}
where $\lfloor x\rfloor$ denotes the integer part of $x$.

An evolution algebra ${\bf E}$ is called \textit{nilpotent} if
there exists some $n\in\mathbb N$ such that ${\bf E}^m=\bf 0$. The
smallest $m$ such that ${\bf E}^m=\bf 0$ is called the index of
nilpotency.

\begin{thm}\label{thm_ro}\cite{rozomir}
An $n$-dimensional evolution algebra $\bf E$ is nilpotent iff it admits a natural basis such that the matrix of the
structural constants corresponding to $\bf E$ in this basis is represented in the form
$$
\tilde{A}=\left(
\begin{array}{lllll}
0 & \tilde{a}_{12} & \tilde{a}_{13} & \vdots  & \tilde{a}_{1n}\\
0 & 0 & \tilde{a}_{23} & \vdots  & \tilde{a}_{2n}\\
\vdots & \vdots & \vdots & \ddots & \vdots\\
0 & 0 & 0 & \vdots & \tilde{a}_{n-1,n}\\
0 & 0 & 0 & \vdots & 0
\end{array}\right)
$$
\end{thm}
Due to Theorem \ref{thm_ro} any nilpotent evolution algebra $\bf{E}$ with $\dim({\bf{E}}^2)=n-1$ has the following form:
\begin{equation}\label{evolalg}
{\bf e}_i^2=\left\{
\begin{array}{lll}
\sum\limits_{j=i+1}^na_{ij}{\bf e}_j, & i\leq n-1;\\
{\bf 0}, & i=n.
\end{array}
\right.
\end{equation}
where $a_{ij}\in\mathbb K$ and $a_{i,i+1}\neq0$ for any $i<n$.

In \cite{clor} it has been establish that a nilpotent evolution
algebra has maximal index of nilpotency $2^{n-1} + 1$, if and only
if the multiplication table of $\bf E$ is given by
\eqref{evolalg}.

In what follows, we restrict ourselves to the nilpotent evolution
algebras with maximal index of nilpotency.

%
%

\section{Derivations and automorphisms}

In this section, we recall some auxiliary facts from \cite{MKOQ19}
which will be used in the next section.

Recall that derivation of an evolution algebra $\bf{E}$ is a linear mapping
$d : \bf{E}\to\bf{E}$ such that
$d({\bf u}{\bf v}) = d({\bf u}){\bf v} + {\bf u}d({\bf v})$
for all ${\bf u}, {\bf v}\in\bf{E}$.

We note that for any algebra, the space $\Der(\bf{E})$ of all derivations is a Lie algebra w.r.t.
the commutator multiplication:
$$
[d_1,d_2]=d_1d_2-d_2d_1,\ \ \ \forall d_1,d_2\in\Der({\bf E}).
$$

%

For a given structural matrix $A=(a_{ij})_{i,j\geq1}^n$ of nilpotent evolution algebra
$\bf E$ with $dim({\bf E}^2)=n-1$ we denote
\begin{equation}\label{A_a_ijneq0}
I_A=\{(i,j):i+1<j<n,\ a_{ij}\neq0\}.
\end{equation}

\begin{thm}\label{thm_der}\cite{MKOQ19}
Let $\bf{E}$ be an evolution algebra with structural matrix $A=(a_{ij})_{i,j\geq1}^n$ in a
natural basis
$\{{\bf e}_i\}_{i=1}^n$. If $\bf E$ is a nilpotent with
$rank A=n-1$, then the following statements hold
\begin{enumerate}
\item[$(i)$] if $I_A\neq\emptyset$ then
$$
\Der({\bf{E}})=\left\{
\left(
\begin{array}{lllll}
0 & 0 & \ldots & 0 & \beta\\
0 & 0 & \ldots & 0 & 0\\
\vdots & \vdots & \ddots & \vdots & \vdots\\
0 & 0 & \ldots & 0 & 0\\
0 & 0 & \ldots & 0 & 0
\end{array}
\right):\ \beta\in\mathbb K
\right\}
$$
\\

\item[$(ii)$]
if $I_A=\emptyset$ then
$$
\Der({\bf{E}})=\left\{
\left(
\begin{array}{lllll}
\alpha & 0 & \ldots & 0 & \beta\\
0 & 2\alpha & \ldots & 0 & (2-2^{n-1})\alpha a_{1n}\\
\vdots & \vdots & \ddots & \vdots & \vdots\\
0 & 0 & \ldots & 2^{n-2}\alpha & (2^{n-2}-2^{n-1})\alpha a_{n-2,n}\\
0 & 0 & \ldots & 0 & 2^{n-1}\alpha
\end{array}
\right):\ \alpha,\beta\in\mathbb K
\right\}
$$
\end{enumerate}
\end{thm}

Recall that by an {\it automorphism} of an evolution algebra
$\bf E$ we mean an isomorphism of $\bf E$ into itself. The set of all
automorphisms is denoted by $Aut({\bf E})$. It is known that
$Aut({\bf E})$ is a group. In this section we are going
to describe $Aut({\bf E})$ of nilpotent evolution algebras
with maximal index of nilpotency.

If $I_A\neq\emptyset$, then by $\eta$ we denote the largest common divisor
of all numbers $2^{j-1}-2^{i}$ where $(i,j)\in I_A$, i.e.,
\begin{equation}\label{eta}
\eta=LCD_{(i,j)\in I_A}(2^{j-1}-2^i)
\end{equation}

\begin{thm}\label{thm_auto}\cite{MKOQ19}
Let $\bf E$ be an $n$-dimensional nilpotent evolution algebra with maximal index of nilpotency
and $A=(a_{ij})_{i,j=1}^n$ be its structural matrix in a natural basis $\{{\bf e}_i\}_{i=1}^n$.
Then the following statements hold:
\begin{enumerate}
\item[$(i)$] if $I_A\neq\emptyset$ then
$$
Aut({\bf{E}})=\left\{
\left(
\begin{array}{lllll}
\alpha & 0 & \ldots & 0 & \beta\\
0 & \alpha^2 & \ldots & 0 & \varphi_{2n}\\
\vdots & \vdots & \ddots & \vdots & \vdots\\
0 & 0 & \ldots & \alpha^{2^{n-2}} & \varphi_{n-1,n}\\
0 & 0 & \ldots & 0 & \alpha^{2^{n-1}}
\end{array}
\right):\ \alpha,\beta\in\mathbb K,\ \alpha^\eta=1
\right\}
$$
where $\eta$ is defined as \eqref{eta}, and $\varphi_{in}$ is given by the following recurrence formula
$$
\begin{array}{ll}
\varphi_{n-1,n}=a_{n-2,n}(\alpha^{2^{n-2}}-\alpha^{2^{n-1}}),\\[2mm]
\varphi_{n-i,n}=a_{n-i-1,n}(\alpha^{2^{n-i-1}}-\alpha^{2^{n-1}})-\sum\limits_{k=1}^{i-1}a_{n-i-1,n-k}\varphi_{n-k,n}, & 1<i<n-1.
\end{array}
$$
\item[$(ii)$] if $I_A=\emptyset$ then
$$
Aut({\bf{E}})=\left\{
\left(
\begin{array}{lllll}
\alpha & 0 & \ldots & 0 & \beta\\
0 & \alpha^2 & \ldots & 0 & a_{1,n}(\alpha^2-\alpha^{2^{n-1}})\\
\vdots & \vdots & \ddots & \vdots & \vdots\\
0 & 0 & \ldots & \alpha^{2^{n-2}} & a_{n-2,n}(\alpha^{2^{n-2}}-\alpha^{2^{n-1}})\\
0 & 0 & \ldots & 0 & \alpha^{2^{n-1}}
\end{array}
\right):\ \alpha,\beta\in\mathbb K,\ \alpha\neq0
\right\}
$$
\end{enumerate}
\end{thm}

\section{On the set $\exp(\Der({\bf E}))$}

In this section, we are going to study the structure of the set $Exp(Der({\bf{E}})$.
Assume that ${\bf{E}}$ be an evolution algebra with a structural matrix $A=(a_{ij})_{i,j\geq1}^n$ in a
natural basis $B=\{{\bf e}_i\}_{i=1}^n$. In \cite{t} it was considered an $\ell_1$-norm on  ${\bf{E}}$ with respect to $B$ which is defined as follows
$$
\|\xb\|_1=\sum_{k=1}^n|x_k|
$$
whenever $\xb=\sum_{k=1}^nx_k{\bf e}_k$. However, it is noted in \cite{MV19} that the defined norm is not consistent with evolution algebra
multiplication. Furthermore, the authors have found necessary and sufficient conditions on the algebra when the norm $\|\cdot\|_1$ is an algebra norm. In this section, we are going to define a norm which endows the considered evolution algebra with norm algebra without any condition.

Let us denote
\begin{equation}\label{FF2}
\gamma:=\sup_{1\leq j\leq n}\bigg\{\sum_{i=1}^n|a_{ij}|\bigg\}
\end{equation}
and define
\begin{equation}\label{normx}
\|{{\bf x}}\|_\gamma =\gamma\max_{i\in I}\{|x_i|\},
\end{equation}
here, as before,
$$
{\bf x}=\sum_{i=1}^nx_{i}{\bf e}_i.
$$

\begin{thm} Let ${\bf{E}}$ be an evolution algebra with a structural matrix $A=(a_{ij})_{i,j\geq1}^n$ in a
natural basis $B=\{{\bf e}_i\}_{i=1}^n$. Then the pair $({\bf{E}}, \|\cdot\|_\gamma)$ is a Banach algebra.
\end{thm}

\begin{proof} It is obvious that $\|\cdot\|_\g$ is a norm. Now, we establish that
$\parallel\cdot\parallel_\g$ is an algebra norm. Indeed, let $\xb=\sum_{i=1}^nx_i{\bf e}_i$ and
$\yb=\sum_{i=1}^ny_i{\bf e}_i$. Then
\begin{eqnarray}\label{prod}
{\bf x}\cdot{\bf y}=\sum_{i,j=1}^na_{ij}x_iy_i{\bf e}_j.
\end{eqnarray}
By \eqref{prod}, \eqref{normx}, we have
\begin{eqnarray*}
\|{\bf x}\cdot{\bf y}\|_\g&=&\gamma\max_{1\leq j\leq n}\bigg\{\bigg|\sum_{i=1}^na_{ij}x_i y_i\bigg|\bigg\}\\
&\leq&\gamma^2\max_{i}\{|x_i|\}\max_{i}\{|y_i|\}\\
&=&\|{\bf x}\|_\g\|{\bf y}\|_\g.
\end{eqnarray*}
The finite dimensionality of ${\bf{E}}$ implies that it is a Banach space, then we arrive at the desired assertion.
\end{proof}

\begin{rk}
The defined norm provides further directions in the infinite dimensional evolution algebras which involves $\ell_\infty$-norms. In this
point, we mention recent works on Hilbert norms on infinite dimensional evolution algebras \cite{VCR20}.
\end{rk}

According to the previous theorem, for any linear mapping $D$ of ${\bf{E}}$ we may introduce its exponential as follows
$$
e^D=\sum_{k\geq0}\frac{D^k}{k!}
$$
where the convergence is considered with respect to the norm $\|\cdot\|_\g$.

\begin{rk}
 We point out that in most of the literature, $e^D$ is defined for nilpotent mappings. However, the Banach algebra structure allows us to investigate such operators in general setting.
\end{rk}

For a given evolution algebra ${\bf E}$,  we set
$$
\exp\left(\Der({\bf{E}})\right):=\left\{
e^d: d\in\Der({\bf{E}})
\right\}.
$$

\begin{problem} Describe the set $\exp\left(\Der({\bf{E}})\right)$.
\end{problem}

A main result of this section is to partially solve this problem in the class of nilpotent evolution algebras with structural matrix $A=(a_{ij})_{i,j\geq1}^n$ such that $\textrm{rank}(A)=n-1$.

\begin{pro}\label{prop_exp}
Let $\bf{E}$ be an evolution algebra with structural matrix $A=(a_{ij})_{i,j\geq1}^n$ in a
natural basis
$\{{\bf e}_i\}_{i=1}^n$. If $\bf E$ is a nilpotent with
$\textrm{rank}(A)=n-1$, then the following statements hold:
\begin{enumerate}
\item[$(i)$] if $I_A\neq\emptyset$ then
$$
\exp\left(\Der({\bf{E}})\right)=\left\{
\left(
\begin{array}{lllll}
1 & 0 & \ldots & 0 & \beta\\
0 & 1 & \ldots & 0 & 0\\
\vdots & \vdots & \ddots & \vdots & \vdots\\
0 & 0 & \ldots & 1 & 0\\
0 & 0 & \ldots & 0 & 1
\end{array}
\right):\ \beta\in\mathbb K
\right\}
$$
\\

\item[$(ii)$]
if $I_A=\emptyset$ then
$$
\exp\left(\Der({\bf{E}})\right)=\left\{
\left(
\begin{array}{lllll}
e^\alpha & 0 & \ldots & 0 & \beta\\
0 & e^{2\alpha} & \ldots & 0 & (e^{2\alpha}-e^{2^{n-1}\alpha})a_{1n}\\
\vdots & \vdots & \ddots & \vdots & \vdots\\
0 & 0 & \ldots & e^{2^{n-2}\alpha} & (e^{2^{n-2}\alpha}-e^{2^{n-1}\alpha})a_{n-2,n}\\
0 & 0 & \ldots & 0 & e^{2^{n-1}\alpha}
\end{array}
\right):\ \alpha,\beta\in\mathbb K
\right\}
$$
\end{enumerate}
where $I_A$ is given by \eqref{A_a_ijneq0}.
\end{pro}

\begin{proof}
$(i)$ Let $I_A\neq\emptyset$. Then due to Theorem \ref{thm_der} for any derivation $d$
it holds $d^2=0$. So,
$$
e^d=I+d=\left(
\begin{array}{lllll}
1 & 0 & \ldots & 0 & \beta\\
0 & 1 & \ldots & 0 & 0\\
\vdots & \vdots & \ddots & \vdots & \vdots\\
0 & 0 & \ldots & 1 & 0\\
0 & 0 & \ldots & 0 & 1
\end{array}
\right).
$$
$(ii)$ Let $I_A=\emptyset$. Pick any derivation $d$. Thanks to Theorem \ref{thm_der} one can find
$\alpha,\beta\in\mathbb K$ such that
$$
d=\left(
\begin{array}{lllll}
\alpha & 0 & \ldots & 0 & \beta\\
0 & 2\alpha & \ldots & 0 & (2-2^{n-1})\alpha a_{1n}\\
\vdots & \vdots & \ddots & \vdots & \vdots\\
0 & 0 & \ldots & 2^{n-2}\alpha & (2^{n-2}-2^{n-1})\alpha a_{n-2,n}\\
0 & 0 & \ldots & 0 & 2^{n-1}\alpha
\end{array}
\right).
$$
After some basic calculations one gets
$$
d^m=\left(
\begin{array}{lllll}
\alpha^m & 0 & \ldots & 0 & \frac{2^{m(n-1)}-1}{2^{n-1}-1}\alpha^{m-1}\beta\\
0 & 2^m\alpha^m & \ldots & 0 & (2^m-2^{m(n-1)})\alpha^ma_{1n}\\
\vdots & \vdots & \ddots & \vdots & \vdots\\
0 & 0 & \ldots & 2^{m(n-2)}\alpha^m & (2^{m(n-2)}-2^{m(n-1)})\alpha^ma_{n-2,n}\\
0 & 0 & \ldots & 0 & 2^{m(n-1)}\alpha^m
\end{array}
\right),\ \ \ \ \forall m\geq2.
$$
Keeping in mind the last one, we obtain
\begin{equation}\label{der33}
e^d=\left(
\begin{array}{lllll}
e^\alpha & 0 & \ldots & 0 & \beta'\\
0 & e^{2\alpha} & \ldots & 0 & (e^{2\alpha}-e^{2^{n-1}\alpha})a_{1n}\\
\vdots & \vdots & \ddots & \vdots & \vdots\\
0 & 0 & \ldots & e^{2^{n-2}\alpha} & (e^{2^{n-2}\alpha}-e^{2^{n-1}\alpha})a_{n-2,n}\\
0 & 0 & \ldots & 0 & e^{2^{n-1}\alpha}
\end{array}
\right),
\end{equation}
where $\beta'=\frac{e^{2^{n-1}\alpha}-e^{\alpha}}{(2^{n-1}-1)\alpha}\beta$, and we mean
$\frac{e^{2^{n-1}\alpha}-e^{\alpha}}{(2^{n-1}-1)\alpha}=1$ for $\alpha=0$. Due to the arbitrariness of
$\beta$ we conclude that every derivation has a form \eqref{der33}.
This completes the proof.
\end{proof}

\begin{thm}
Let $\bf{E}$ be an evolution algebra with structural matrix $A=(a_{ij})_{i,j\geq1}^n$ in a
natural basis
$\{{\bf e}_i\}_{i=1}^n$. If $\bf E$ is a nilpotent with
$rank(A)=n-1$. Then $\exp\left(\Der({\bf E})\right)$ is a normal subgroup of
$Aut({\bf E})$.
\end{thm}

\begin{proof}
It is easy to check that $\exp\left(\Der({\bf E})\right)$ is a normal subgroup of
$Aut({\bf E})$ when $I_A\neq\emptyset$. So, we consider only the case $I_A=\emptyset$.

Let us assume that $I_A=\emptyset$. Then according to assumptions $(ii)$ of Proposition \ref{prop_exp} and
Theorem \ref{thm_auto}, we see that $\exp\left(\Der({\bf E})\right)$ is subset of $Aut({\bf E})$.
To complete the proof it is enough to check: $e^{d_1}e^{d_2}\in\exp\left(\Der({\bf E})\right)$
for any pair $d_1,d_2\in\Der({\bf E})$.

Let us take two derivations:
$$
d_k=\left(
\begin{array}{lllll}
\alpha_k & 0 & \ldots & 0 & \beta_k\\
0 & 2\alpha_k & \ldots & 0 & (2-2^{n-1})\alpha_k a_{1n}\\
\vdots & \vdots & \ddots & \vdots & \vdots\\
0 & 0 & \ldots & 2^{n-2}\alpha_k & (2^{n-2}-2^{n-1})\alpha_k a_{n-2,n}\\
0 & 0 & \ldots & 0 & 2^{n-1}\alpha_k
\end{array}
\right),\ \ \ \ k=1,2.
$$
Then
\begin{equation}\label{eq33}
e^{d_1}e^{d_2}=\left(
\begin{array}{llllll}
e^{\alpha_1+\alpha_2} & 0 & \ldots & 0 & \lambda(\alpha_1,\beta_1,\alpha_2,\beta_2)\\
0 & e^{2(\alpha_1+\alpha_2)} & \ldots & 0 & \nu_1(\alpha_1,\alpha_2)a_{1n}\\
\vdots & \vdots & \ddots & \vdots & \vdots\\
0 & 0 & \ldots & e^{2^{n-2}(\alpha_1+\alpha_2)} & \nu_{n-2}(\alpha_1,\alpha_2)a_{n-2,n}\\
0 & 0 & \ldots & 0 & e^{2^{n-1}(\alpha_1+\alpha_2)}
\end{array}
\right)
\end{equation}
where
\begin{equation}\label{eq34}
\lambda(\alpha_1,\beta_1,\alpha_2,\beta_2)=\frac{e^{2^{n-1}\alpha_2}(e^{2^{n-1}\alpha_1}-e^{2\alpha_1})}{3\alpha_1}\beta_1+
\frac{e^{2^{n-1}\alpha_1}(e^{2^{n-1}\alpha_2}-e^{2\alpha_2})}{3\alpha_2}\beta_2,
\end{equation}
\begin{equation}\label{eq35}
\nu_k(\alpha_1,\alpha_2)=(e^{2^k(\alpha_1+\alpha_2)}-e^{2^{n-1}(\alpha_1+\alpha_2)}),\ \ \ \ 1\leq k\leq n-2.
\end{equation}
Denoting $\alpha:=\alpha_1+\alpha_2$ and
$\beta:=\frac{(2^{n-1}-1)\alpha}{e^{2^{n-1}\alpha}-e^\alpha}\lambda(\alpha_1,\beta_1,\alpha_2,\beta_2)$
due to Theorem \ref{thm_der} we have the following derivation
$$
d=\left(
\begin{array}{lllll}
\alpha & 0 & \ldots & 0 & \beta\\
0 & 2\alpha & \ldots & 0 & (2-2^{n-1})\alpha a_{1n}\\
\vdots & \vdots & \ddots & \vdots & \vdots\\
0 & 0 & \ldots & 2^{n-2}\alpha & (2^{n-2}-2^{n-1})\alpha a_{n-2,n}\\
0 & 0 & \ldots & 0 & 2^{n-1}\alpha
\end{array}
\right).
$$
Thanks to Proposition \ref{prop_exp} one has
\begin{equation}\label{eq36}
e^d=\left(
\begin{array}{lllll}
e^\alpha & 0 & \ldots & 0 & \frac{e^{2^{n-1}\alpha}-e^{\alpha}}{(2^{n-1}-1)\alpha}\beta\\
0 & e^{2\alpha} & \ldots & 0 & (e^{2\alpha}-e^{2^{n-1}\alpha})a_{1n}\\
\vdots & \vdots & \ddots & \vdots & \vdots\\
0 & 0 & \ldots & e^{2^{n-2}\alpha} & (e^{2^{n-2}\alpha}-e^{2^{n-1}\alpha})a_{n-2,n}\\
0 & 0 & \ldots & 0 & e^{2^{n-1}\alpha}
\end{array}
\right)
\end{equation}
Finally, noticing \eqref{eq34},\eqref{eq35} from \eqref{eq33},\eqref{eq36} we obtain
$$
e^{d_1}e^{d_2}-e^d=0,
$$
which means that $\exp\left(\Der({\bf E})\right)$ is a subgroup of
$Aut({\bf E})$. If we pick two automorphisms $\varphi\in Aut({\bf E})$ and $\phi\in\exp\left(\Der({\bf E})\right)$
then  $\varphi\phi\varphi^{-1}\in\exp\left(\Der({\bf E})\right)$. The arbitrariness
of $\varphi$ and $\phi$ implies that $\exp\left(\Der({\bf E})\right)$ is a subgroup of
$Aut({\bf E})$.
The proof is complete.
\end{proof}

\begin{ex}
Let ${\bf E}$ be a Banach evolution algebra with the structural matrix $A=\left(a_{ij}\right)_{i,j\geq1}^n$
in a natural basis $\{{\bf e}_1,\dots,{\bf e}_n\}$. For a given $D\in\Der({\bf E})$ we consider
homogeneous differential equation of the form
\begin{equation}\label{dif_eq}
\dot{\xb}=D\xb,
\end{equation}
where $\xb(t)=\sum_{i=1}^nx_i(t){\bf e}_i$.
It is known that solutions of \eqref{dif_eq} has the following form $\xb(t)=e^{tD}\xb(0)$.
When ${\bf E}$ is nilpotent with $rankA=n-1$ then thanks to Proposition \ref{prop_exp} all solutions of \eqref{dif_eq}
has the following form:
$$
\xb(t)=\sum_{i=1}^{n-1}e^{2^{i-1}\alpha t}x_i(0)+\left(\frac{e^{2^{n-1}\alpha t}-e^{\alpha t}}{(2^{n-1}-1)\alpha t}+e^{2^{n-1}\alpha t}
+\sum_{i=1}^{n-2}(e^{2^i\alpha t}-e^{2^{n-1}\alpha t})a_{i,n}\right)x_n(0).
$$
\end{ex}
Let us take two non abelian subgroups of $GL_2(\mathbb K)$:
$$
H_1=\left\{\left(
\begin{array}{ll}
\alpha & \beta\\
0 & \alpha^2
\end{array}
\right): \alpha,\beta\in\mathbb K,\ \alpha^\eta=1\right\},\ \ \
H_2=\left\{\left(
\begin{array}{ll}
\alpha & \beta\\
0 & \alpha^2
\end{array}
\right): \alpha,\beta\in\mathbb K,\ \alpha\neq0\right\},
$$
where $\eta$ is defined as \eqref{eta}.
\begin{pro}\label{pro_iso_exp}
Let $\bf{E}$ be an evolution algebra with structural matrix $A=(a_{ij})_{i,j\geq1}^n$ in a
natural basis
$\{{\bf e}_i\}_{i=1}^n$. If $\bf E$ is a nilpotent with
$rank(A)=n-1$, then $Aut({\bf E})$ is isomorphic to
$$
H:=\left\{
\begin{array}{ll}
H_1, & \mbox{if } I_A\neq\emptyset,\\
H_2, & \mbox{if } I_A=\emptyset.
\end{array}
\right.
$$
\end{pro}
\begin{proof}
Let us suppose that $I_A=\emptyset$. First of all we notice that
$Aut({\bf E})$ is isomorphic to the group
$H_3$ formed by all $2\times2$ matrices $\left(
\begin{array}{ll}
\alpha & \beta\\
0 & \alpha^{2^{n-1}}
\end{array}
\right)$ ($\alpha,\beta\in\mathbb K$, $\alpha\neq0$). So it is enough to establish that $H_3$ and $H_2$
are isomorphic which established by the following mapping $\varphi: H_2\to H_3$ defined by
$$
\varphi:\left(
\begin{array}{ll}
a & b\\
0 & a^2
\end{array}\right)\mapsto\left(
\begin{array}{ll}
a^{\mu_n} & a^{\mu_n-1}b+\frac{1}{2}a^{\mu_n+1}-\frac{1}{2}a^{\mu_n}\\
0 & a^{\mu_n+1}
\end{array}\right),\ \ \ \ \mu_n=\frac{1}{2^{n-1}-1},\ n>2.
$$
The proof of case $I_A\neq\emptyset$ is similar to the proof of $I_A=\emptyset$.
\end{proof}
\begin{cor}\label{cor_exp}
Let $\bf{E}$ be an evolution algebra with structural matrix $A=(a_{ij})_{i,j\geq1}^n$ in a
natural basis
$\{{\bf e}_i\}_{i=1}^n$. If $\bf E$ is a nilpotent with
$rank(A)=n-1$, then $\exp\left(\Der({\bf E})\right)$ is isomorphic to
$$
H':=\left\{
\begin{array}{ll}
H'_1, & \mbox{if } I_A\neq\emptyset,\\
H'_2, & \mbox{if } I_A=\emptyset,
\end{array}
\right.
$$
where groups $H'_1$ and $H'_2$ formed by all
$2\times2$ matrices $\left(
\begin{array}{ll}
1 & \beta\\
0 & 1
\end{array}
\right)$ and $\left(
\begin{array}{ll}
e^\alpha & \beta\\
0 & e^{2\alpha}
\end{array}
\right)$ ($\alpha,\beta\in\mathbb K$) respectively.
\end{cor}
Thanks to Proposition \ref{pro_iso_exp} and Corollary \ref{cor_exp}
we obtain the following result.
\begin{pro}
Let $\bf{E}$ be an evolution algebra with structural matrix $A=(a_{ij})_{i,j\geq1}^n$ in a
natural basis
$\{{\bf e}_i\}_{i=1}^n$. If $\bf E$ is a nilpotent with
$rank(A)=n-1$, then the following statements hold:\\
\begin{enumerate}
\item[$(i)$] if $I_A\neq\emptyset$ then
$Aut({\bf E})/_{\exp\left(\Der({\bf E})\right)}\cong G/_{\{1\}}$. Here $G$ is a group of
$\eta$-th roots of unity and $\eta$ is defined by $\eqref{eta}$.
\\
\item[$(ii)$]
if $I_A=\emptyset$ then
$Aut({\bf E})/_{\exp\left(\Der({\bf E})\right)}\cong\mathbb K^*/_{\exp(\mathbb K)}$.
Here $\mathbb K^*=\mathbb K\setminus\{0\}$ and $\exp(\mathbb K)$ is a range of exponential function on $\mathbb K$.
\end{enumerate}
\end{pro}

\section*{Acknowledgments}
The present work is supported by the UAEU UPAR Grant Code: G00003447.

\end{document}